\documentclass[12pt]{article}

\usepackage[english]{babel}
\usepackage{amsfonts}
\usepackage{epsfig}
\usepackage{color}
\usepackage{amsmath}
\usepackage{amsthm}
\usepackage{amssymb}
\usepackage{cancel}
\usepackage{soul}
\usepackage{float}
\usepackage{comment}
\usepackage{verbatim}
\usepackage{xcolor}
\usepackage[ruled,vlined]{algorithm2e}
\usepackage{url}
\usepackage{soul}

\newcommand{\ba}{\begin{array}}
	\newcommand{\bal}{\begin{array}{l}}
		\newcommand{\be}{\begin{equation}}
			\newcommand{\beqa}{\begin{eqnarray}}
				\newcommand{\bl}{\begin{lem}}
					\newcommand{\bt}{\begin{teo}}
						
						\newcommand{\C}{\mathbb{C}}

						\newcommand{\ea}{\end{array}}
					\newcommand{\ee}{\end{equation}}
				\newcommand{\eeqa}{\end{eqnarray}}
			\newcommand{\el}{\end{lem}}
		\newcommand{\et}{\end{teo}}

	\newcommand{\kxn}{k[x_1,\dots,x_n]}
	
	\newcommand{\la}{\langle}
	
	\newcommand{\mc}{\mathcal}

	\newcommand{\N}{\mathbb{N}}
	
\newcommand{\nk}{\mathcal{N}_\lambda^{(k)}}

	\newcommand{\Q}{\mathbb{Q}}

	\newcommand{\ra}{\rangle}

	\newcommand{\vnj}{\mathcal{V}^n_j}

	\newcommand{\xb}{{\bf x}}

	\newcommand{\yb}{{\bf y}}
	
	\newcommand{\Z}{\mathbb{Z}}

	\newtheorem{lem}{Lemma}
	\newtheorem{thm}{Theorem}
	\newtheorem{pro}{Proposition}

	\theoremstyle{definition}                 
	\newtheorem{exstar}{Example}             

\begin{document}

\title{{First integrals and invariants of  systems of ODEs}}
\author{Mateja Gra\v si\v c$^{1,2}$, Abdul Salam Jarrah$^3$ and Valery G.~Romanovski$^{1,4,5}$\\
$^1${\it Faculty of Natural Science and Mathematics,}\\ {\it  University of Maribor,
Koro\v ska cesta 160, SI-2000 Maribor, Slovenia}\\
$^2${\it Institute of Mathematics, Physics and Mechanics,}\\
{\it Jadranska 19, SI-1000 Ljubljana, Slovenia}\\
$^3${\it Department of Mathematics and Statistics}\\{\it American University of Sharjah,  Sharjah, UAE}\\
$^4${\it Faculty of Electrical Engineering and Computer Science,} \\ {\it University of Maribor,
 Koro\v ska cesta 46, SI-2000 Maribor, Slovenia}\\
$^5${\it Center for Applied Mathematics and Theoretical Physics,}\\
{\it Mladinska 3, SI-2000 Maribor, Slovenia}
}
\date{}

\maketitle

\begin{abstract}

 We investigate the interplay between monomial first integrals, polynomial invariants of certain group action, and the Poincaré–Dulac normal forms for  autonomous systems of ODEs with diagonal matrix of the linear part. Using tools from computational algebra, we develop an algorithmic approach for identifying generators of the algebras of monomial and polynomial first integrals, which works  in the general case  where the matrix of the linear part includes algebraic complex eigenvalues. Our method also provides a practical tool for exploring the algebraic structure of polynomial invariants and their relation to the Poincar\'e-Dulac  normal forms of the underlying vector fields. 
\end{abstract}

\section{Introduction}		





First integrals and invariants are fundamental tools in the study of differential equations. They allow us to better understand the qualitative behavior of solutions by revealing deep structural properties of the system. In particular, first integrals—functions that remain constant along solution trajectories—help identify conserved quantities such as energy, momentum, or other geometric invariants. These conserved quantities can simplify the analysis of a system, reduce the number of effective variables, or even allow the system to be integrated completely in special cases (see e.g. \cite{A,AC,Bibikov,Bru,LPW,YZ,Zhang} and the references given there).
Invariants, in general, provide information about the symmetries and stability of the system. 
The presence of invariants is often key to the classification of differential systems up to certain equivalence relations.


The connection between monomial first integrals and Poincaré-Dulac normal forms has been studied extensively in \cite{KWZ, LPW, St, WalcherNF}. The theory of polynomial invariants for parametric families of ODEs was developed by Sibirsky and his school \cite{Sib1, Sib2}.  These invariants are important for the classification of ODEs \cite{ALSV} and for studies related to the center-focus problem \cite{Liu2, RS}.

In this paper, we study monomial first integrals, polynomial invariants, 
and their connection to Poincar\'e-Dulac normal forms
for the $n$-dimensional autonomous system 
		\be
		\label{Asn}
		\dot x = Ax+  X(x), \hspace{3mm} x \in \C^n, 
		\ee
		where  $A$ is a complex diagonal matrix,
        \be \label{Adiag}
A={\rm diag}(\lambda_1,\dots, \lambda_n),
\ee
and \(X(x)\) is a power series without constants or linear terms.

Let $\mathcal{V}^n$ be the space of polynomial vector fields $v: \mathbb{C}^n \rightarrow \mathbb{C}^n$ and, for $j >1$, let  $\vnj \subseteq \mathcal{V}^n$ be the subspace of vector fields $v$ whose components $v_i$ are homogeneous polynomials of degree $j$, for $i=1,\dots, n$. 
It is not difficult to see   that  any   formal invertible  change of coordinates of the form 
		\be
		\label{yhy}
		x = y+ \  H(y)= y +\sum_{j=2}^\infty  H_j(y), 
		\ee
		with  $ H_{j} \in 
		\mathcal{V}^n_{j}$ for all $j\geq 2$, 
		brings system (vector field) \eqref{Asn} to a system of a similar form, 
		\be
		\label{linearni}
		\dot y = A y+ G(y),
		\ee
		where 
		$$ G(y)= \sum_{j=2}^\infty \ G_j(y), \quad \mbox{ and }  G_j(y)\in \vnj \quad \mbox{ for all } j \geq 2.
		$$ 

Let $\mathbb{N}_0 = \mathbb{N}\cup \{0\}$ be the set of nonnegative integers. For $\alpha=(\alpha_1,\dots,\alpha_n) \in \mathbb{N}_0^n$, we define $x^\alpha := \prod_{i=1}^n x_i^{\alpha_i}$, where $x=(x_1,\dots,x_n)$ is a column vector. Let ${ e}_k$ be the $n$-dimensional row unit vector for $k=1,\ldots, n$. A term $a y^\alpha$  in ${ e}_k G(y)$ or $bx^\alpha$ in $e_kX(x)$ is called \emph{resonant} if 
		\be \label{res_c_e} 
		\la \lambda , \alpha\ra  - \lambda_k=0,
		\ee
where $\lambda=(\lambda_1,\dots, \lambda_n)$ and  $\la \lambda, \alpha\ra $ is  the usual inner product of $n$-tuples.
We denote by $\nk$ the set of all solutions $\alpha\in \N_0^n$ to \eqref{res_c_e}, 
that is 
\be \label{nk}
\nk=\{  \alpha\in \N_0^n \ : \ \la \lambda ,  \alpha\ra  - \lambda_k=0   \}. 
\ee 
System \eqref{linearni} is said to be in the \emph{Poincar\'e--Dulac normal form}, or simply in normal form, if $G(y)$ contains only resonant terms.
By the Poincar\'e-Dulac theorem \cite{P,Dul}, any system \eqref{Asn} can be transformed to a normal form by a suitable transformation \eqref{yhy} \cite{P,Dul}. In particular, transformation \eqref{yhy}, which brings system \eqref{Asn} to a  Poincar\'e--Dulac normal form,  is called  a  normalization or a normalizing transformation.



Equivalently (see e.g. \cite{AFG,LPW,WalcherNF}), 
system \eqref{linearni}  is in the Poincar\'e–Dulac normal form if $ [Ay,G(y)] = 0,$ 
that is,  for all $j\ge 2$, $[Ay,G_j(y)] = 0.$

Following \cite{KWZ}, we define the centralizers
\be \label{ker}
C_0^{for} (A) :=\{  g\in  \C[[x_1,\dots,  x_n]]^n \ :  [g, A] = 0\}
\ee 
and 
$$
C^{pol} (A) :=\{  g\in  \C[x_1,\dots,  x_n]^n \ :  [g, A] = 0\}.
$$
By Lemma 2.9 of \cite{KWZ}  and Corollary 4.5.9 of \cite{Mur}, $C_0^{for}(A)$ is spanned as a $\C$-vector space by all $ x^\alpha e_k$  with $\alpha\in \mathcal{N}_\lambda^{(k)}$.
To our knowledge, at present there are no algorithmic methods available to compute a generating set of 
$C_0^{for}(A)$ for a given $A$
-- some particular cases are discussed in \cite{KWZ,Mur}. 

Recall that 
a  polynomial $f \in \kxn$  (where $k$ is a field) is  invariant under the action of the group $\mathfrak{G}$ (or simply {an invariant of $\mathfrak{G}$}) if $f({ x}) = f(\mathfrak{g} { x})$  for every ${ x} \in k^n$ and every $\mathfrak{g} \in \mathfrak{G}$. 

In this paper, 
we first deal with monomial first integrals of the linear approximation of \eqref{Asn},  that is, of the system
 \begin{equation}\label{Als}
\dot x =A x, 
\end{equation}
and propose an algorithm to compute the generators
of the algebra of monomial  first integrals of \eqref{Als}.
Subsequently, we establish a connection between specific monomial invariants of polynomial systems \eqref{Asn} and first integrals of linear differential systems. Furthermore, we adapt our algorithm to compute generators of the algebra of invariants for \eqref{Asn}, applicable even when some eigenvalues  are not rational but algebraic elements. In Section \ref{sec:nf}, we demonstrate how our algorithm effectively describes the structures of  $C_0^{for}(A)$ and $C^{pol}(A)$.

This work contributes to the broader understanding of polynomial invariants, monomial first integrals, and the structure of Poincar\'e-Dulac normal forms of ordinary differential equations. 

\section{Preliminaries}

For $A$ defined by \eqref{Als}, let $\lambda=(\lambda_1, \dots, \lambda_n)$, where $\lambda_1, \dots, \lambda_n$ are the diagonal elements of $A$,  
and denote
\begin{equation}\label{efre}
\mathcal{M}_\lambda:=\left\{\alpha \in \N_0^n\  | \  \langle\lambda,\alpha\rangle=0
\right\} = \left\{\alpha \in \N_0^n\  | \ A \alpha^{\top}=0\right\}.
\end{equation}
Clearly, 
 $\mathcal{M}_\lambda$  is a submonoid of the additive monoid $\Z^n$.

It is obvious that the elements $\alpha=(\alpha_1,\dots, \alpha_n)$ of $\N_0^n$
are in one-to-one correspondence with the monic monomials $x^\alpha=x_1^{\alpha_1}\cdots x_n^{\alpha_n}$
of the polynomial ring $\C[x_1,\dots, x_n]$. Furthermore, $ \alpha\in \mathcal{M}_\lambda$
if and only if the monomial  $x^\alpha$ is a first integral of  system \eqref{Als}.

 Let $I(\lambda)$ be the algebra of polynomial first integrals of \eqref{Als}
and   $\mathcal{R}_\lambda$ be  the $\mathbb{Z}$-module spanned by the
elements of $\mathcal{M}_\lambda$.
By the results of \cite{LPW,WalcherNF}, we have the following statement.

\begin{lem}\label{lem:LPW}
 The algebra $I(\lambda)$ of polynomial first integrals of system \eqref{Als}
  is finitely generated $\C$-algebra. Furthermore, if  
  $\text{rank}(\mathcal{R}_\lambda)=d$, then there are exactly $d$ functionally independent polynomial first
integrals of \eqref{Als}.
\end{lem}

\begin{lem}\label{lem2n}
   The monoid $\mathcal{M}_\lambda $  has a unique generating set $H_\lambda$ (called the Hilbert basis of $\mathcal{M}_\lambda$).
\end{lem}
\begin{proof}

{By Lemma \ref{lem:LPW}, the algebra of polynomial first integrals
of \eqref{Als} is finitely generated, so the additive monoid $\mathcal{M}_\lambda$
is also finitely generated. As a submonoid of ${\mathbb Z}^n$ with only one unit element, namely $\bar 0=(0,\dots,0)$, $\mathcal{M}_\lambda$ is the so-called pointed affine monoid (see \cite{MS} for definition). By 
Proposition 7.15 of \cite{MS} it has a unique minimal finite generating set.}
\end{proof}

The vector monomials $x^\alpha e_k ,\alpha \in \nk$, and the elements of the $\C$-vector space they generate are often called equivariants (see \cite{Mur}). 
 By Proposition 1.6 of \cite{WalcherNF} (see also Lemma 4.2 of \cite{KWZ}),
 $C^{pol}(A) $ is a finitely generated $I(\lambda)$-module.
Thus there arises an important problem to study the structure of $C^{pol}(A) $ (see e.g. \cite{KWZ,Mur,WalcherNF}).
By {Lemma 4.4} of \cite{KWZ}, if $\lambda_i\ne 0$ $(i=1,\dots,n)$ and all $\alpha\in \nk$ ($k=1,\dots, n$)  have non-negative entries, then $C^{pol}(A)$ is a free $I(\lambda)$-module of rank $n$ generated by $v_j(x)=x_j e_j$ where $j=1,\dots,n$. However, in general  the structure of $C^{pol}(A)$ can be rather complicated, and,
 to describe it,
one needs to find generators for $\mathcal{M}_\lambda$ and {describe the sets} $\nk$, $k=1,\dots, n.$ 
We consider this problem in Section \ref{sec:nf}. In particular, in this section  we present an algorithm to compute the equivariants of system \eqref{Als} and a unique presentation of $C^{pol}(A)$.

In the next section, we propose an algorithm to compute a Hilbert basis of $\mathcal{M}_\lambda$
and a generating set of the $\mathbb{Z}$-module $\mathcal{R}_\lambda$. \bigskip

\section{Monomial first integrals}\label{sec_DioEqu}
Obviously,
a monomial $x^\alpha$, $\alpha\in \Z^n $, of the ring of Laurent polynomials in $x_1, \dots, x_n$, is a first integral  of linear system \eqref{Als}
if and only if 
\be \label{dio}
\la \alpha,  \lambda\ra =0.
\ee
Such a monomial $x^\alpha$ is called a  Laurent monomial first integral. 
The monomials corresponding to solutions $\alpha\in \N_0^n$ of equation \eqref{dio}
generate the algebra of polynomial first integrals of \eqref{Als}.

By Lemma \ref{lem:LPW}, the algebra of the polynomial first integrals  of  system \eqref{Als} is finitely generated. Since a polynomial $r(x)$ is a first integral of \eqref{Als} if and only if each of its monomials is a first integral of \eqref{Als}, to describe the algebra it is sufficient to find monomial first integrals of \eqref{Als}. By Lemma \ref{lem2n} the latter algebra has a unique generating set. 
In this section, we consider the problem of how to find the generators of the algebra under some assumptions on $\lambda$. 

Let $\lambda_1,\dots,\lambda_n$ be algebraic elements  of $\C$, that is, each $\lambda_i$ is a root of a polynomial with integer coefficients. Let $K = \Q(\lambda_1,\dots,\lambda_n)$ be the finite algebraic extension of $\Q$ containing $\lambda_1,\dots,\lambda_n$. Let $b_1,\dots,b_d$ be a basis of $K$ over $\Q$. Then each $\lambda_i$ is a $\Q$-linear combination of $b_j$'s. That is, 
\begin{equation}\label{lc_lambda}
\lambda_i = \sum_{j=1}^d c_{ji}b_j
\end{equation}
where $c_{ji} \in \Q$. Let $C=[c_{ji}]$ be the $d \times n$ matrix whose $ji$-entry is $c_{ji}$ and  $l = lcm(c_{ji} : j=1,\dots,d \mbox{ and }i=1,\dots,n) $.
Let $\mathfrak{A} = lC=[lc_{ji}]$ be the integer $d \times n$-matrix obtained from $C$ by clearing the denominators. In particular,
 \begin{equation}\label{matA}
 \mathfrak{A}=[\mathfrak{a}_1 \, \mathfrak{a_2} \cdots \mathfrak{a}_n], 
\end{equation}
 where $\mathfrak{a}_i$ represents the $i$-th column of the matrix and corresponds to $\lambda_i$.

\begin{lem}\label{lem:u}
For $\alpha = (\alpha_1,\alpha_2,\dots,\alpha_n) \in \Z^n$,
\[
\langle\alpha,\lambda\rangle=\sum_{i=1}^n \lambda_i \alpha_i =0 \iff \mathfrak{A}\alpha=0.
\]
\end{lem}
\begin{proof}
    We use the notations presented in the paragraph preceding the lemma. Using the basis $\{b_1,\ldots,b_d \}$ of $K$ over $\mathbb Q$ and the expressions \eqref{lc_lambda}, we have $\langle\alpha,\lambda\rangle=\sum_{j=1}^d\left(\sum_{i=1}^n c_{ji}\alpha_i\right)b_j$. 
    {Since $b_1,\dots, b_d$ are independent over $\Q$, it yields} that $ \langle\alpha,\lambda\rangle=0$ if and only if
    \begin{equation*}
        \sum_{i=1}^nc_{ji}\alpha_i=0\ \text{ for all }\ j=1,\ldots,d,
    \end{equation*}
    which is equivalent to $C\alpha^\top=0$ and also  $\mathfrak{A}\alpha^\top = lC\alpha^\top=0$.
\end{proof}
By Lemma \ref{lem:u},
the set $\left\{\alpha \in \N_0^n \ | \ \mathfrak{A} \alpha^\top=0\right\}$ is the same as the monoid $\mathcal{M}_\lambda$ defined by \eqref{efre} . 

Notice that  
\be\label{kr}
K_\rho = \left\{ \alpha \in \mathbb{Z}^n \ \middle| \ \mathfrak{A} \alpha^\top = 0 \right\}
\ee  
is a subgroup of the finitely generated abelian group \(\mathbb{Z}^n\). Therefore, \(K_\rho\) is itself a finitely generated abelian group, and its basis can be readily computed, for example, using Smith normal form, as the kernel of the group homomorphism  
\[
\rho : \mathbb{Z}^n \longrightarrow \mathbb{Z}^d, \quad \rho(\nu) = \mathfrak{A} \nu.
\]
Furthermore, the monoid  
\[
\mathcal{M}_\lambda = \mathbb{N}_0^n \cap K_\rho
\]  
consists of the nonnegative integer solutions in the kernel.
However, computing a minimal generating set (or minimal spanning set) $\mathcal{H}_\lambda$ for $\mathcal{M}_\lambda$ is more subtle and generally requires tools from integer linear programming  or commutative algebra, such as Gr\"obner bases or Hilbert basis techniques \cite{LDE1,LDE2, LDE3}.

{
We note that the following statement takes place.\begin{pro} \label{pro:1n}
  If the rank of $\ker \rho$ is equal to $p$ and the rank of $\mathcal{R}_\lambda$ is also $p$, then both  $K_\rho$ and $\mathcal{R}_\lambda$ are generated by $p$ linearly independent elements of $H_\lambda$. \end{pro}
}

The Hilbert basis $H_\lambda$ of the monoid ${\mathcal M}_\lambda$ can be obtained using Algorithm \ref{alg1}, which employs Algorithm 1.4.5 of \cite{St-AIT}.

 \begin{algorithm}[ht]
\KwIn{A vector $\lambda = (\lambda_1,\dots,\lambda_n) \in\C^n$ of algebraic elements} {
\KwOut{Hilbert basis $H_\lambda$ of  ${\mathcal M}_\lambda$ and a minimal spanning set of $\mathcal{R}_\lambda$}
}
\nl Find a number field $K$ containing all $\lambda_i$

\nl Choose a $\Q-$basis $\{b_1,\dots,b_d\}$ for $K$

\nl Express each $\lambda_i$ in the chosen basis: 
Write 
\[
\lambda_i = c_{1i}b_1 +c_{2i}b_2+ \cdots + c_{di}b_d,  \mbox{ where }c_{ji} \in \Q
\]

\nl Let $C=[c_{ji}]$ be the $d\times n$-matrix of coefficients 

\nl  Convert to integer matrix: Let  $\mathfrak{A}= [a_{ji}] = lC=[lc_{ji}]$ be the integer $d \times n$-matrix obtained from $C$ by clearing the denominators

\nl   Compute the reduced Gr\"obner basis $\mathcal G$ for the ideal 
$$
 \left \langle x_i -  z_i \displaystyle\prod_{j=1}^d t_j^{a_{ji}}\ : \ i=1,\dots, n\right\rangle
$$     
in ${k}[t_1,\ldots,t_d,t_1^{-1},\ldots,t_d^{-1},x_1,\dots,x_n, z_1,\ldots,z_n]$  with respect to $\succ$:
\[
\{t_1^\pm,\ldots,t_d^\pm\}\succ \{x_1,\ldots,x_n\}\succ \{z_1,\ldots,z_n\}.
\]

\nl  The Hilbert basis $H_\lambda$ of ${\mathcal M}_\lambda$ consists of all vectors $\nu$ such that $\mathbf{x}^{\nu}-\mathbf{z}^{\nu}$ appears in $\mathcal G$.

\nl  The monomials $x^\nu $ such that $\mathbf{x}^{\nu}-\mathbf{z}^{\nu}$ appears in $\mathcal G$ generate $I(\lambda)$.

\caption{\bf Hilbert basis of ${\mathcal M}_\lambda $ and generators of $I(\lambda)$  \label{alg1}}
\end{algorithm}

\begin{exstar} \label{ex1}
    Let $\lambda = (1,\zeta, \zeta^2,-2,3)$ where $\zeta$ is the third root of unity. 
    The corresponding system \eqref{Als} is 
\be\label{ex1_sys}
\dot x ={\rm diag}(1,\zeta, \zeta^2,-2,3)\, x.
\ee

    Note that  $\zeta^2=-\zeta-1$, and hence, $\Q(\lambda)= \Q(\zeta)$, and $\{1,\zeta\}$ is a basis for $\Q(\zeta)$ over $\Q$. 
Therefore, 
\begin{eqnarray*}
    \lambda_1 &=& 1 =1 \cdot 1 +0 \cdot \zeta \\
    \lambda_2 &=& \zeta = 0 \cdot 1 +1 \cdot \zeta \\
    \lambda_3 &=& \zeta^2 =-1-\zeta = -1\cdot 1 +(-1) \cdot \zeta\\
    \lambda_4 &=& -2= 2\cdot 1 +0 \cdot \zeta \\
    \lambda_5 &=& 3= 3\cdot 1 +0 \cdot \zeta.
\end{eqnarray*}
Hence 
\[
\mathfrak{A}=
 \begin{bmatrix} 
 1 & 0 & -1 & -2 & 3\\ 
 0 & 1 & -1 & 0 & 0 
\end{bmatrix}.
\]
Implementing Algorithm \ref{alg1} in the computer algebra system Macaulay2 \cite{M2}, we obtain the Hilbert basis of ${\mathcal M}_\lambda$ to be
\[
H_\lambda =\{ (0,0,0,3,2), (0,1,1,1,1), (0,3,3,0,1), (1,0,0,2,1), (1,1,1,0,0), (2,0,0,1,0)\}
\]
and the generators of $I(\lambda)$ are, therefore, equal to
\be \label{genI}
 I_1=x_4^3x_5^2,\ I_2= x_2 x_3x_4 x_5, \  I_3 = x_2^3x_3^3x_5, \ I_4 = x_1x_4^2x_5, \ I_5 =x_1x_2x_3, \ I_6 = x_1^2x_4.
\ee 
The  Macaulay2 code and calculations for the example are given in Appendix A. 
\end{exstar}

  It is worth noting here that $\mathcal{R}_\lambda$ may not be free, as $\Z-$syzygies may exist among the elements of $\mathcal{H}_\lambda$. To compute a minimal generating set for the $\mathbb{Z}$-module $\mathcal{R}_\lambda$, classical methods from linear algebra can be used, as in Algorithm \ref{alg2}  presented  below.

\begin{algorithm}[H]
\KwIn{Hilbert basis $H_\lambda$}
\KwOut{Generating set of the $\mathbb{Z}$-module $\mathcal{R}_\lambda$}

\nl Let $N$ be the matrix whose columns are the vectors in the Hilbert basis $H_\lambda$. \\

\nl Compute the column space of $N$ over $\mathbb{Z}$  using Gaussian elimination. \\

\nl The set of columns of $N$ corresponding to pivot columns in the reduced echelon form forms a generating set for the $\mathbb{Z}$-module $\mathcal{R}_\lambda$.

\caption{{\bf Generating Set of the $\mathbb{Z}$-Module $\mathcal{R}_\lambda$} \label{alg2}}
\end{algorithm}

In Example \ref{ex1}, the matrix 
\[ 
N=   
\begin{bmatrix} 
0 & 0 & 0 & 1 & 1 & 2\\ 
0 & 1 & 3 & 0 & 1 & 0\\ 
0 & 1 & 3 & 0 & 1 & 0\\ 
3 & 1 & 0 & 2 & 0 & 1\\ 
2 & 1 & 1 & 1 & 0 & 0\\ 
 \end{bmatrix},
\] 
obtained from the Hilbert basis, after the Gaussian elimination, is reduced to
\[
\begin{bmatrix} 
1 &  0 & -1 & 0 & -1 & -1 \\
0 & 1 &  3 & 0 & 1 &  0\\
0 & 0 & 0 & 1 & 1&  2 \\
0 & 0 & 0 & 0 & 0 & 0\\ 
0 & 0 & 0 & 0 & 0 & 0\\ 
 \end{bmatrix}.
\]
Notice that columns 1, 2, and 4 have the leading terms and hence  the $\mathbb{Z}$-module $\mathcal{R}_\lambda$ is of rank 3 and is generated by
\[
\{ (0,0,0,3,2), (0,1,1,1,1), (1,0,0,2,1)\}.
\]

Since the number of generators of ${\mathcal R}_\lambda$ is 3, system \eqref{ex1_sys}
has 3 functionally  independent polynomial first integrals, which can be chosen to be $I_1,I_2,I_4$ .  The calculation of the reduced row echelon form of matrix $N$ was done in SAGE \cite{sage}, see Appendix B.

 \section{Invariants} \label{sec:inv}
			

In this section, we study polynomial invariants of $n$-dimensional polynomial systems of ODEs under the action of a one-parameter group on the phase space of the system. 

Let   ${\bf N}_i=\{Q=(q_1, \dots, q_n)\in \Z^n : q_i\ge -1, q_j\ge 0 {\ \rm if \ } j\ne i     \}  $ and  ${\bf N}={\bf N}_1\cup \dots \cup {\bf N}_n  $.  Following Bruno \cite{Bru},  we can write any  $n$-dimensional analytical or formal system of ODEs in the form 
\be\label{sys_g}
\dot x = \sum_{Q\in {\bf N}}  (x\odot  a_Q) x^Q,
\ee
where the $\odot$ stands for the Hadamard product,
\be \label{aq}
a_Q=(a^{(Q)}_1, \dots,  a^{(Q)}_n )^\top \in \C^n. 
\ee

 More specifically,  any  $n$-dimensional polynomial system can be  written as 
\be\label{sys_p}
\dot x =  \sum_{Q\in \Omega} (x \odot a_Q) x^Q,
\ee
where $\Omega $ is a  finite   set, say 
$$
\Omega=\{ Q_1, \dots, Q_\ell\} \subseteq {\bf N}.
$$

The matrix $A={\rm diag}\, \lambda, $
where $\lambda =\left(\lambda_1,\lambda_2,\ldots,\lambda_n \right)$
 is the infinitesimal generator of the complex matrix group 
$e^{A\phi}$.
After the change of variables 
\be \label{Ax}
y=e^{A\phi}x ,
\ee
 we obtain from \eqref{sys_p} the  system
 $$
 \dot y=y\odot \sum_Q a_Q e^{-\la \lambda, Q\ra \phi } y^Q =  y\odot \sum_Q \widehat  a_Q  y^Q .
 $$
In particular, 
the coefficients are changed by the rule
\be\label{ch3d}
a_{Q}\mapsto  \widehat a_Q= a_{Q} e^{-\la \lambda, Q\ra \phi}, 
\ee
where, as above,  $\la \lambda, Q\ra$ represents the usual inner product of $n$-tuples $\lambda$ and $Q$.

Let  $\mathfrak{m}=n  \ell$ be  the number of parameters $a_i^{(Q)}$ in \eqref{sys_p}.
Define the ordered $\mathfrak{m}$-tuple  of parameters of system \eqref{sys_p}
as
\begin{equation} \label{Apar}
{a}=\left( a_{Q_1}^\top,\ldots, a_{Q_\ell}^\top\right)= \left( a_1^{(Q_1)}, \dots, a_n^{(Q_1)},\dots, a_1^{(Q_\ell)},\dots, a_n^{(Q_\ell)}
  \right)
\end{equation}
and consider the algebra of complex polynomials $\C[ a]$, where the variables 
are the parameters of  system \eqref{sys_p}.
Let 
$$
\nu(i)=(\nu^{(i)}_1, \dots ,  \nu^{(i)}_n), \quad i=1,\dots, \ell
$$
and $$
\nu=(\nu(1), \dots, \nu(\ell)).
$$

For $a_{Q}$ defined by \eqref{aq} and $\alpha=(\alpha_1,\dots, \alpha_n)$,  let 
$$
\left( a^{(Q)}\right)^\alpha:=\left( a_1^{(Q)}\right)^{\alpha_1}  \left( a_2^{(Q)}\right)^{\alpha_2} \cdots  \left( a_n^{(Q)}\right)^{\alpha_n}. 
$$
Given  the ordered 
$\mathfrak m$-tuple ${a}$,  
the  monomial $a^\nu$ in ${\mathbb C}[a]$ is  defined by 
\begin{equation*}
a^\nu= \prod_{i=1}^\ell a_{Q_i}^{\nu(i)}.
\end{equation*}
The change of variables \eqref{Ax} (the group action \eqref{ch3d}) induces a $\mathbb C$-linear map on ${\mathbb C}[a]$ acting on monomial $a^\nu$ as 
\be \label{ant}
a^\nu \mapsto \prod_{i=1}^\ell \widehat{ a}_{Q_i}^{\nu(i)} = \prod_{i=1}^\ell { a}_{Q_i}^{\nu(i)} e^{-\sum_{k=1}^n\la \lambda, Q_i\ra   \nu_k^{(i)} } = \prod_{i=1}^\ell { a}_{Q_i}^{\nu(i)} e^{-\la \lambda, \mathcal{Q}_i 
\nu(i)^\top \ra }   ,
\ee
where
$\mathcal{Q}_i$ is $n\times  n $ matrix whose each column is $Q_i.$ Let $\mathcal{Q}$
be the $n\times \mathfrak{m}$
matrix 
$$
\mathcal{Q}=[\mathcal{Q}_1\ \mathcal{Q}_2\ \cdots \mathcal{Q}_\ell]  
$$

and let 
$$
L: \N_0^\mathfrak{m} \to \C^n
$$
be the map
\begin{equation} \label{defL}
L(\nu) := \left(L^1(\nu),L^2(\nu),\ldots,L^n(\nu)\right)^\top
=\mathcal{Q} \nu^\top.
\end{equation}

Using  additive map \eqref{defL}, we  define the monoid
\begin{equation} \label{defM}
\mathcal M_L= \left \{ \nu\in {\mathbb N}_0^{\mathfrak{m}} \ :\ \la \lambda, L(\nu) \ra =  0 \right \},
\end{equation}
which is finitely generated by the same argument as $\mathcal{M}_\lambda$ introduced above. 

\begin{thm} \label{THMinvarM}
A monomial $a^\nu $ is an invariant of  group \eqref{ch3d} if and only if 
$\nu \in {\mc M}_L$. 
\end{thm}
\begin{proof}
From \eqref{ant}, we see that the monomial 
$a^\nu$  is an invariant of \eqref{ch3d} if and only if 
$$
\sum_{i=1}^\ell \la    \lambda, \mathcal{Q}_i 
\nu(i)^\top \ra =0.
$$
But, the above equality is equivalent 
$$
\sum_{i=1}^\ell \la    \lambda, \mathcal{Q}_i 
\nu(i)^\top \ra = \la \lambda, \sum_{i=1}^\ell \mathcal{Q}_i \nu(i)^\top   \ra  = \la \lambda, \mathcal{Q}\nu^\top\ra  = \la \lambda, L(\nu) \ra=   0.
$$
 \end{proof}

The following proposition can easily be verified by straightforward computations. 
\begin{pro}\label{pro:1}
    For $\nu \in {\mathbb N}_0^{\mathfrak{m}}$, 
     $\nu \in {\mc M}_L$ if and only if the monomial $a^\nu$ is a first integral of the linear system 
    $$
    \dot a= \mathcal{L} a,
    $$
    where  
    $$
    \mathcal{L}={\rm  diag}(\lambda^\top \mathcal{Q}).
    $$
\end{pro}

From Lemma \ref{lem:LPW} and Proposition \ref{pro:1},
we conclude that $\mathcal{M}_L$ has a Hilbert basis $H_L$, which can be computed using Algorithm 1 above. 
This algorithm works in the general case where the $\lambda_i$ are algebraic elements of $\C$, and is not limited to the case where all $\lambda_i$  are rational numbers.  In the latter simple case, we may use Algorithm 1.4.5 from \cite{St-AIT}.

\begin{exstar} \label{ex2}
Consider the following differential system, written in the form \eqref{sys_p},
 with linear part as in Example \ref{ex1},
\begin{equation*}
\dot x = x\odot \lambda + x\odot a_{Q_1}\, x^{Q_1}+x\odot a_{Q_2}\, x^{Q_2}+x\odot a_{Q_3}\, x^{Q_3},
\end{equation*}
where  $\lambda = \left(1, \zeta,\zeta^2,-2,3 \right)$  and  $Q_1=(1,0,0,0,1), Q_2=(0,1,1,0,0)$, and $Q_3=(1,0,0,1,1)$.
In particular, the complex vectors of coefficients are
\begin{eqnarray*}
    a_{Q_1} &= (a_{10001}, b_{10001},  c_{10001}, d_{10001},  f_{10001}), \\
   a_{Q_2} &= (a_{01100},b_{01100},  c_{01100}, d_{01100},  f_{01100}), \\
    a_{Q_3} &= (a_{10011},b_{10011},  c_{10011}, d_{10011},  f_{10011}).
\end{eqnarray*}

Following the above notations, let $\mathcal{Q}=[\mathcal{Q}_1\ \mathcal{Q}_2\ \mathcal{Q}_3]$ be the $5\times 15$ matrix
\[
\mathcal{Q} = 
\left[
\begin{array}{ccccccccccccccc}
1 & 1 & 1 & 1 & 1 & 0 & 0 & 0 & 0 & 0 & 1 & 1 & 1 & 1 & 1 \\
0 & 0 & 0 & 0 & 0 & 1 & 1 & 1 & 1 & 1 & 0 & 0 & 0 & 0 & 0 \\
0 & 0 & 0 & 0 & 0 & 1 & 1 & 1 & 1 & 1 & 0 & 0 & 0 & 0 & 0 \\
0 & 0 & 0 & 0 & 0 & 0 & 0 & 0 & 0 & 0 & 1 & 1 & 1 & 1 & 1 \\
1 & 1 & 1 & 1 & 1 & 0 & 0 & 0 & 0 & 0 & 1 & 1 & 1 & 1 & 1 \\
\end{array}
\right]
\]
and
\[
\lambda^\top \mathcal{Q}=
(4 , 4 , 4 , 4 , 4 , -1 , -1 , -1 , -1 , -1 , 2 , 2 , 2 , 2 , 2).
\]
To obtain the Hilbert basis for the monoid $\mathcal M_L$ of solutions $\nu \in \N_0^{15}$, we use  Algorithm \ref{alg1} with the input vector $\lambda^\top \mathcal{Q}$. It turned out that the basis has 425 vectors, which are the exponent vectors of the monomials returned by the algorithm. Since the list is rather long, we do not present it in the paper, but we present the computations for the particular case  when 
\[
b_{10001}=d_{10001}=c_{01100}=d_{01100}=a_{10011}=b_{10011}=d_{10011}=0
.\]
In this case, 211 binomials are  generated by the algorithm. Only 64 binomials are  of the form  $\xb^\nu - \yb^\nu $. The leading terms of the binomials are  
given in Appendix C. Using them, we obtain the Hilbert basis of $\mathcal{M}_L$, which consists of the 64 vectors in $\N_0^{15}$. In Appendix D, we present a Macaulay2 code used to compute the Hilbert basis of $\mathcal M_L$.
 Using this Hilbert basis, we get the following invariants of group action \eqref{ch3d} on our system:
\begin{eqnarray*}
&f_{01100}^2f_{10011}, 
f_{01100}^2c_{10011}, 
d_{01100}, 
b_{01100}f_{01100}f_{10011}, 
b_{01100}f_{01100}c_{10011},\\ 
&b_{01100}^2 f_{10011}, 
b_{01100}^2c_{10011},
a_{01100}f_{01100}f_{10011}, 
a_{01100}f_{01100}c_{10011}, 
a_{01100}b_{01100}f_{10011}, 
a_{01100}b_{01100}c_{10011}, \\
&a_{01100}^2 f_{10011}, 
a_{01100}^2c_{10011}, 
f_{10001}f_{01100}^4, 
f_{10001}b_{01100}f_{01100}^3, 
f_{10001}b_{01100}^2f_{01100}^2, 
f_{10001}b_{01100}^3f_{01100},\\ 
& f_{10001}b_{01100}^4, 
f_{10001}a_{01100}f_{01100}^3, 
f_{10001}a_{01100}b_{01100}f_{01100}^2, 
f_{10001}a_{01100}b_{01100}^2f_{01100}, 
f_{10001}a_{01100}b_{01100}^3, \\
&f_{10001}a_{01100}^2f_{01100}^2, 
f_{10001}a_{01100}^2b_{01100}f_{01100},
f_{10001}a_{01100}^2b_{01100}^2, 
f_{10001}a_{01100}^3f_{01100}, 
f_{10001}a_{01100}^3b_{01100},\\ 
&f_{10001}a_{01100}^4,
c_{10001}f_{01100}^4, 
c_{10001}b_{01100}f_{01100}^3, 
c_{10001}b_{01100}^2f_{01100}^2, 
c_{10001}b_{01100}^3f_{01100}, 
c_{10001}b_{01100}^4, \\
&c_{10001}a_{01100}f_{01100}^3, 
c_{10001}a_{01100}b_{01100}f_{01100}^2, 
c_{10001}a_{01100}b_{01100}^2f_{01100}, 
c_{10001}a_{01100}b_{01100}^3, \\
&c_{10001}a_{01100}^2f_{01100}^2, 
c_{10001}a_{01100}^2b_{01100}f_{01100}, 
c_{10001}a_{01100}^2b_{01100}^2, 
c_{10001}a_{01100}^3f_{01100}, 
c_{10001}a_{01100}^3b_{01100},\\ 
&c_{10001}a_{01100}^4, 
a_{10001}f_{01100}^4, 
a_{10001}b_{01100}f_{01100}^3, 
a_{10001}b_{01100}^2f_{01100}^2, 
a_{10001}b_{01100}^3f_{01100}, 
a_{10001}b_{01100}^4, \\
&a_{10001}a_{01100}f_{01100}^3, 
a_{10001}a_{01100}b_{01100}f_{01100}^2, 
a_{10001}a_{01100}b_{01100}^2f_{01100}, 
a_{10001}a_{01100}b_{01100}^3,\\ 
&a_{10001}a_{01100}^2f_{01100}^2,
a_{10001}a_{01100}^2b_{01100}f_{01100}, 
a_{10001}a_{01100}^2b_{01100}^2,\\ 
&a_{10001}a_{01100}^3f_{01100}, 
a_{10001}a_{01100}^3b_{01100}, 
a_{10001}a_{01100}^4.
\end{eqnarray*}
\end{exstar}

\section{Normal forms}
\label{sec:nf}

In this section, we discuss the structure of  $C_0^{for}(A) $  and $C^{pol}(A)$, as defined by equations \eqref{ker}, and present an algorithm for computing 
$\nk$.


By \cite{KWZ,WalcherNF},  $C^{pol}$  is an $I(\lambda)$-module.
We have an algorithm to describe $I(\lambda)$, so we need to find only a generating set of this $I(\lambda)$-module.

As before, let $\lambda=(\lambda_1, \dots, \lambda_n)^\top $,   ${\bf N}_i=\{Q\in \Z^n : Q_i\ge -1, Q_j\ge 0 {\ \rm if \ } j\ne i     \}  $ and  $\bf N={\bf N}_1\cup \dots \cup {\bf N}_n  $. 
{To find a generating set of equivariants, we need to describe the sets
$$
\nk=\{  \alpha\in \N_0: \ \la \alpha, \lambda\ra -\lambda_k=0\}\quad k=1,\dots, n.
$$
Let 
$$
\widetilde \nk=\{  \beta\in {\bf N}_k: \ \la \beta, \lambda\ra =0\}\quad k=1,\dots, n
.$$
The following statement is obvious.
\begin{pro} 
1) For any $k=1,\dots,n$  the monoid $\mathcal{M}_\lambda$ is a subset of $\widetilde \nk.$\\
2)  There is   one-to-one correspondence between the elements of $\nk$ and elements of  $\widetilde \nk$. 
\end{pro}

To describe  
$\nk$, recall the $d \times n$ matrix $\mathfrak{A}$ from (\ref{matA}), which resulted from the equation  $\sum_{i=1}^n \alpha_i \lambda_i = 0$. To solve $\sum_{i=1}^n \alpha_i \lambda_i = \lambda_k$, we introduce a slack variable $\alpha_0$ and solve the following homogeneous equation  in $(\alpha_0, \alpha_1, \dots, \alpha_n) \in \mathbb{N}_0^{n+1}$:
\begin{equation}\label{alp0}
-\alpha_0\lambda_k + \sum_{i=1}^n \alpha_i \lambda_i = 0.
\end{equation}
This can be accomplished by applying Algorithm~\ref{alg1} to the following
$d \times (n+1)$ matrix:
\begin{equation}\label{matAplus}
\mathfrak{A}_k = [-\mathfrak{a}_k \,\, \mathfrak{a}_1 \,\, \mathfrak{a}_2 \, \cdots \,\, \mathfrak{a}_n],
\end{equation}
where $\mathfrak{a}_j$ denotes the $j$-th column of $\mathfrak{A}$  (\ref{matA}).

In the resulting Hilbert basis, the set of vectors where $\alpha_0 = 0$ is precisely the Hilbert basis $H_\lambda$ of $\mathcal{M}_\lambda$, while those vectors where $\alpha_0 = 1$ are solutions to $\sum_{i=1}^n \alpha_i \lambda_i = \lambda_k$. Let $S$ be the set of all such vectors. It is clear that any solution to equation (\ref{alp0}) is of the form $\nu +\mu$ where $\nu \in S$ and $\mu \in \mathcal{M}_\lambda$.

Algorithm 3 is an extension of Algorithm \ref{alg1}}, incorporating the procedure above and producing a minimal generating set for $\nk$.

 \begin{algorithm}[ht]
\KwIn{A vector $\lambda = (\lambda_1,\dots,\lambda_n) \in\C^n$ of algebraic elements} {
\KwOut{A minimal generating set for $\nk$}
}
\nl Find a number field $K$ containing all $\lambda_i$

\nl Choose a $\Q-$basis $\{b_1,\dots,b_d\}$ for $K$

\nl Express each $\lambda_i$ in the chosen basis: 
Write 
\[
\lambda_i = c_{1i}b_1 +c_{2i}b_2+ \cdots + c_{di}b_d,  \mbox{ where }c_{ji} \in \Q
\]

\nl Let $C= [-\mathfrak{c}_k \,\, \mathfrak{c}_1 \,\, \mathfrak{c}_2 \, \cdots \,\, \mathfrak{c}_n]$ be the $d\times (n+1)$-matrix of coefficients, where $\mathfrak{c}_{i}$ is the $d-$column vector  $\mathfrak{c}_{i} = 
 (c_{1i}, c_{2i}, \dots, c_{di})$

\nl  Convert to integer matrix: Let  $\mathfrak{A}_k= [-\mathfrak{a}_k \,\, \mathfrak{a}_1 \,\, \mathfrak{a}_2 \, \cdots \,\, \mathfrak{a}_n]$
 be the integer $d \times (n+1)$-matrix obtained from $\mathfrak{C}$ by clearing the denominators

\nl   Compute the reduced Gr\"obner basis $\mathcal G$ for the ideal 
$$
 \left \langle x_0 - z_0 \displaystyle\prod_{j=1}^d t_j^{-a_{jk}}, \, x_i -  z_i \displaystyle\prod_{j=1}^d t_j^{a_{ji}}\ : \ i=1,\dots, n\right\rangle
$$     
in ${k}[t_1,\ldots,t_d,t_1^{-1},\ldots,t_d^{-1},x_0, x_1,\dots,x_n, z_0, z_1,\ldots,z_n]$  with respect to $\succ$:
\[
\{t_1^\pm,\ldots,t_d^\pm\}\succ \{x_0, x_1,\ldots,x_n\}\succ \{z_0, z_1,\ldots,z_n\}.
\]

\nl  The Hilbert basis $H_\lambda$ of ${\mathcal M}_\lambda$ consists of all vectors $\mu$ such that $x_1^{\mu_1}\cdots x_n^{\mu_n} - z_1^{\mu_1}\cdots z_n^{\mu_n}$ appears in $\mathcal G$.

\nl  Let $S$ be the set of all vector $\nu$ such that $x_0x_1^{\nu_1}\cdots x_n^{\nu_n} - z_0z_1^{\nu_1}\cdots z_n^{\nu_n}$ appears in $\mathcal G$.

\nl Every vector in $\nk$ is of the form $\nu+\mu$ for some $\nu \in S$ and $\mu \in  {\mathcal M}_\lambda$.

\caption{\bf Minimal generating set  of $\nk$  \label{alg3}}
\end{algorithm}

\begin{exstar} \label{ex3}
Consider system \eqref{ex1_sys} from Example \ref{ex1}. Employing Algorithm \ref{alg3} as shown in Appendix E, results in the following:
\begin{eqnarray*}
    \mathcal{N}_\lambda^1 &=&\{\nu +\mu : \nu \in \{(1,0,0,0,0),(0,0,0,1,1),(0,2,2,0,1)\} \mbox{ and }
\mu \in \mathcal{M}_\lambda\} \\
\mathcal{N}_\lambda^2 &=&\{(0,1,0,0,0) +\mu : 
\mu \in \mathcal{M}_\lambda\} \\
\mathcal{N}_\lambda^3 &=&\{(0,0,1,0,0) +\mu : \mu \in \mathcal{M}_\lambda\} \\
\mathcal{N}_\lambda^4 &=&\{\nu +\mu : \nu \in \{(0,0,0,1,0),(0,2,2,0,0)\} \mbox{ and }
\mu \in \mathcal{M}_\lambda\} \\
\mathcal{N}_\lambda^5 &=&\{\nu +\mu : \nu \in \{(0,0,0,0,1),(3,0,0,0,0)\} \mbox{ and }
\mu \in \mathcal{M}_\lambda\} \\
\end{eqnarray*}
In particular, in addition to the obvious equivariants of \eqref{ex1_sys}:
$$
v_i=x_i e_i \quad i=1, \dots, 5, 
$$
there are four  additional equivariants:
$$
v_6= x_4 x_5 e_1, \quad v_7= x_2^2 x_3^2 e_4 \quad v_8=x_1^3 e_5, \quad v_9=x_2^2x_3^2 x_5 e_1 .
$$
Let $I_1, \dots, I_6$  be the monomials corresponding to the elements of $H_\lambda$
which are given in \eqref{genI}. Then every element of $C^{pol}(A)$ can be written as 
\begin{multline}\label{st_de}
x_1 e_1 f_1(I_1,\dots,I_6)+x_2 e_2 f_2(I_1,\dots,I_6)+ x_3 e_3 f_3(I_1,\dots,I_6)+  x_4 e_4 f_4(I_1,\dots,I_6) +  x_5 e_5 f_5(I_1,\dots,I_6)+\\x_4 x_5 e_1 f_6(I_1,\dots,I_6)+
x_2^2 x_3^2 e_4 f_7(I_1,\dots,I_6)+
x_1^3 e_5 f_8(I_1,\dots,I_6)+   
x_2^2x_3^2x_5 f_9(I_1,\dots,I_6),
\end{multline}
for some $f_1,\dots,f_9 \in \C[I_1,\dots, I_6]$.
However, it is easy to verify the following  syzygies:
\begin{eqnarray*}
I_4 v_6-I_1 v_1 &=&0,\quad I_5v_6- I_2v_1=0,  \quad I_6v_6- I_4v_1=0 \\
I_1v_7-I_2^2v_4 &=& 0, \quad I_2 v_7-I_3v_4=0,I_4v_7-I_2I_5v_4=0, \quad I_6v_7-I_5^2v_5=0\\
I_1v_8 -I_4I_6v_5 &=& 0 \quad I_2v_8-I_5I_6v_5=0, \quad I_3v_8-I_6^2v_5=0\\
I_1v_9-I_2^2v_1 &=& 0, \quad I_2v_9-I_3v_6=0, \quad I_4v_9-I_2^2v_1=0, \quad I_5v_9-I_3v_1=0 .
\end{eqnarray*}
Using those sysygies, we can write \eqref{st_de}
in a unique way in the  simpler form 
\begin{multline*}
 f_1(I_1,\dots,I_6) v_1 + f_2(I_1,\dots,I_6)v_2+  f_3(I_1,\dots,I_6)v_3+   f_4(I_1,\dots,I_6)v_4+\\ f_5(I_1,\dots,I_6) v_5+
 f_6(I_1,I_2,I_3) v_6+  f_7(I_3,I_5) v_7+ f_8(I_4,I_5,I_6) v_8+  f_9(I_3,I_6) v_9.
\end{multline*}
This leads to the following representation of $C^{pol}$:
\begin{multline*}
C^{pol}(A)={\mathbb C}[I_1,\dots,I_6]v_1\oplus {\mathbb C}[I_1,\dots,I_6]v_2\oplus {\mathbb C}[I_1,\dots,I_6]v_3\oplus {\mathbb C}[I_1,\dots,I_6]v_4 \oplus \\{\mathbb C}[I_1,\dots,I_6] v_5 \oplus
 {\mathbb C}[I_1,I_2,I_3] v_6\oplus  {\mathbb C}[I_3,I_5] v_7\oplus {\mathbb C}[I_4,I_5,I_6] v_8 \oplus {\mathbb C}[I_3,I_6] v_9.
\end{multline*}
By replacing the polynomial ring in the above formula with the ring of formal power series, we obtain a representation of $C^{for}_0$, which  
is called a {Stanley decomposition} of the normal form module \cite{Mur}.
\end{exstar}

Suppose that system \eqref{Asn} is in the normal form and system  \eqref{Als}
admits $p$ independent polynomial first integrals.
By Proposition 5 of \cite{LPW},  system \eqref{Asn}  admits $p$ independent formal first
integrals if and only if it admits every polynomial first integral of  \eqref{Als}. The following statement provides a slight extension of this result. 

\begin{pro} \label{pr4} Let $\phi: \Z^n\to \C  $ be defined by $\phi(x)=\la \lambda, x\ra. $ Suppose that 
\be \label{rfr}
{\rm rank} \, \ker \phi={\rm rank}\, \mathcal{R}_\lambda=p,
\ee
and  system \eqref{Asn}  is in the normal form.  Then system \eqref{Asn} 
admits $p$ functionally independent formal  first integrals if and only if it admits every  monomial first integral of  \eqref{Als} in the ring of Laurent polynomials.
\end{pro}
\begin{proof} First we note that the Laurent monomials $x^{\alpha_1}, \dots, x^{\alpha_m}$ are functionally independent if and only if their exponents  $\alpha_1,\dots,\alpha_m$ are
linearly independent over $\C$. Assume system \eqref{Asn}, which is in the normal form, admits every Laurent  monomial first integral of \eqref{Als}. 
Since ${\rm rank}\,\mathcal{R}_\lambda=p $, there are $\alpha_1, \dots, \alpha_p\in \N_0^n$, which are linearly independent over $\Z$.  Then $\alpha_1, \dots, \alpha_p $ are also linearly independent over $\C$.
Therefore, $x^{\alpha_1}, \dots, x^{\alpha_p}$ are functionally independent  monomial first  integrals, which can also be considered as formal first integrals. 

Suppose now that system \eqref{Asn}, which is in the normal form, admits $p$ formal first integrals. Then \eqref{Als} admits $p$ polynomial first integrals.  
By Proposition 5 of \cite{LPW},
Since ${\rm rank}\, \mathcal{R}_\lambda=p$,  system \eqref{Asn} admits all monomial first integrals of \eqref{Als}. 
By \eqref{rfr},
every solution $\alpha$ to \eqref{dio} is a $\Z$-linear combination of elements of $\mathcal{H}_\lambda$, yielding that the corresponding $x^\alpha$ is a Laurent monomial first integral of \eqref{Asn} (which is in the normal form).
\end{proof}

In general, checking if \eqref{rfr} holds can be a difficult problem. However, in the case where all eigenvalues of $A$ are algebraic elements over {$\Q$},  both $ \mathcal{R}_\lambda$ and its rank can be computed using Algorithm 2. Instead of computing the rank of
$\ker \phi$, we can easily compute the rank of $K_\rho $ defined by \eqref{kr}. In particular, for system \eqref{ex1_sys} of  Example \ref{ex1}, we have  ${\rm rank} \,  K_\rho={\rm rank}\, \mathcal{R}_\lambda=3$. Therefore, by Proposition \ref{pr4}, if  a normal form of an analytic or formal system with the linear part \eqref{ex1_sys} admits 3 independent first integrals, then it admits every Laurent monomial first integral of system \eqref{ex1_sys}.


\section{Conclusions}

We have proposed an algorithm to compute 
the generating set of the algebra of polynomial first integrals of system \eqref{Als}, and 
we have adapted it to compute the generating set of polynomial   invariants of system \eqref{sys_g} 
under action \eqref{Ax}. Furthermore, we applied this approach to study the structure of the Poincar\'e-Dulac normal form of system \eqref{Asn}.
An interesting and  important related  problem is to study the connection of $\nk$ to the  Diophantine hull of $A$ \cite{St}.

\section*{Acknowledgments}	The first and the third  authors are supported by the Slovenian Research and Innovation  Agency (core research programs P1-0288 and P1-0306, respectively) and  by the project  101183111-DSYREKI-HORIZON-MSCA-2023-SE-01 “Dynamical Systems and Reaction Kinetics Networks”.

\section*{Appendix A}
Computations for the Hilbert basis $H_{\lambda}$ and a generating set of $I(\lambda)$ in Example \ref{ex1} were done in Macaulay2 using Algorithm \ref{alg1}. Note that for the Hilbert basis only resulting binomials of the form $\bf{x}^{\nu}-\bf{y}^{\nu}$ are considered.  

\begin{verbatim}
Macaulay2, version 1.24.11-1695-gf35df1017f (vanilla)
with packages:ConwayPolynomials,Elimination,IntegralClosure,InverseSystems,..

K = ZZ/31991; 
K[n,t,x_1..x_5,y_1..y_5, MonomialOrder => Lex];
h = ideal (x_1-y_1*n, x_2-y_2*t,n*t*x_3-y_3,n^2*x_4-y_4,x_5-y_5*n^3);
H = gens gb h
Q = entries H
toString Q

{{x_4^3*x_5^2-y_4^3*y_5^2, x_2*x_3*y_4^2*y_5-x_4^2*x_5*y_2*y_3, 
x_2*x_3*x_4*x_5-y_2*y_3*y_4*y_5, x_2^2*x_3^2*y_4-x_4*y_2^2*y_3^2, 
x_2^3*x_3^3*x_5-y_2^3*y_3^3*y_5, x_1*y_4*y_5-x_4*x_5*y_1, 
x_1*y_2^2*y_3^2*y_5-x_2^2*x_3^2*x_5*y_1, x_1*x_4*y_2*y_3-x_2*x_3*y_1*y_4, 
x_1*x_4^2*x_5-y_1*y_4^2*y_5, x_1*x_2*x_3-y_1*y_2*y_3, 
x_1^2*y_2*y_3*y_5-x_2*x_3*x_5*y_1^2, x_1^2*x_4-y_1^2*y_4, x_1^3*y_5-x_5*y_1^3, 
t*y_2-x_2, t*x_3*y_4^2*y_5-x_4^2*x_5*y_3, t*x_3*y_1*y_4-x_1*x_4*y_3, 
t*x_3*x_5*y_1^2-x_1^2*y_3*y_5, t*x_3*x_4*x_5-y_3*y_4*y_5, 
t*x_2*x_3^2*y_4-x_4*y_2*y_3^2, t*x_2*x_3^2*x_5*y_1-x_1*y_2*y_3^2*y_5, 
t*x_2^2*x_3^3*x_5-y_2^2*y_3^3*y_5, t*x_1*x_3-y_1*y_3, t^2*x_3^2*y_4-x_4*y_3^2, 
t^2*x_3^2*x_5*y_1-x_1*y_3^2*y_5, t^2*x_2*x_3^3*x_5-y_2*y_3^3*y_5, 
t^3*x_3^3*x_5-y_3^3*y_5, n*y_4*y_5-x_4*x_5, n*y_3^2*y_5-t^2*x_3^2*x_5, 
n*y_1-x_1, n*x_4*y_3-t*x_3*y_4, n*x_4^2*x_5-y_4^2*y_5, n*x_2*x_3-y_2*y_3, 
n*x_1*y_3*y_5-t*x_3*x_5*y_1, n*x_1*x_4-y_1*y_4, n*x_1^2*y_5-x_5*y_1^2, 
n*t*x_3-y_3, n^2*y_3*y_5-t*x_3*x_5, n^2*x_4-y_4, n^2*x_1*y_5-x_5*y_1, 
n^3*y_5-x_5}} 

\end{verbatim}

\section*{Appendix B}

To obtain the generating set of the $\mathbb Z$-module ${\mathcal R}_{\lambda}$ from Example \ref{ex1} Algorithm \ref{alg2} is used. The calculations of the reduced row echelon form of matrix $N$, that is needed in the algorithm, were done in SAGE and are presented below.   

\begin{verbatim}
N = matrix(ZZ,[[0,0,0,3,2],[0,1,1,1,1],[0,3,3,0,1],[1,0,0,2,1],[1,1,1,0,0],
[2,0,0,1,0]])
N.column_space()

Free module of degree 6 and rank 3 over Integer Ring
Echelon basis matrix:
[ 1  0 -1  0 -1 -1]
[ 0  1  3  0  1  0]
[ 0  0  0  1  1  2]
\end{verbatim}

\section{Appendix C}
Below is the list of monomials of Example \ref{ex2} when 
\[
b_{10001}=d_{10001}=c_{01100}=d_{01100}=a_{10011}=b_{10011}=d_{10011}=0.
\]
In this case, 
\[
\lambda^\top \mathcal{Q}=
(4 , 0 , 4 , 0 , 4 , -1 , -1 , 0 , 0 , -1 , 0 , 0 , 2 , 0 , 2).
\]
First, using Algorithm \ref{alg1} for the input vector $\lambda^{\top}{\mathcal Q}$, 211  binomials were generated. Only 64  binomials of the form $\bf{x}^{\nu}-\bf{y}^{\nu}$ and 
the leading terms  $\bf{x}^{\nu}$ are next. See Appendix D below for the M2 session that was used to generate this output. 

\begin{verbatim} 
x_14, x_12, x_11, x_10^2*x_15, x_10^2*x_13, x_9, x_8, x_7*x_10*x_15, 
x_7*x_10*x_13, x_7^2*x_15, x_7^2*x_13, x_6*x_10*x_15, x_6*x_10*x_13, 
x_6*x_7*x_15, x_6*x_7*x_13, x_6^2*x_15, x_6^2*x_13, x_5*x_10^4, x_5*x_7*x_10^3, 
x_5*x_7^2*x_10^2, x_5*x_7^3*x_10, x_5*x_7^4, x_5*x_6*x_10^3, x_5*x_6*x_7*x_10^2, 
x_5*x_6*x_7^2*x_10, x_5*x_6*x_7^3, x_5*x_6^2*x_10^2, x_5*x_6^2*x_7*x_10, 
x_5*x_6^2*x_7^2, x_5*x_6^3*x_10, x_5*x_6^3*x_7, x_5*x_6^4, x_4, x_3*x_10^4, 
x_3*x_7*x_10^3, x_3*x_7^2*x_10^2, x_3*x_7^3*x_10, x_3*x_7^4, x_3*x_6*x_10^3, 
x_3*x_6*x_7*x_10^2, x_3*x_6*x_7^2*x_10, x_3*x_6*x_7^3, x_3*x_6^2*x_10^2, 
x_3*x_6^2*x_7*x_10, x_3*x_6^2*x_7^2, x_3*x_6^3*x_10, x_3*x_6^3*x_7, 
x_3*x_6^4, x_2, x_1*x_10^4, x_1*x_7*x_10^3, x_1*x_7^2*x_10^2, x_1*x_7^3*x_10, 
x_1*x_7^4, x_1*x_6*x_10^3, x_1*x_6*x_7*x_10^2, x_1*x_6*x_7^2*x_10, x_1*x_6*x_7^3, 
x_1*x_6^2*x_10^2, x_1*x_6^2*x_7*x_10, x_1*x_6^2*x_7^2, x_1*x_6^3*x_10, 
x_1*x_6^3*x_7, x_1*x_6^4
\end{verbatim}

\section*{Appendix D}

Computations for the Hilbert basis $H_{\lambda}$ and a generating set of $I(\lambda)$ in Example \ref{ex2} were done in Macaulay2 using Algorithm \ref{alg1}. The output of this calculations are 211 binomials. Note that for the Hilbert basis only resulting binomials of the form $\bf{x}^{\nu}-\bf{y}^{\nu}$ have to be considered of which there are 64 binomials, and in Appendix C, we present the list of their leading monomials $\bf{x}^{\nu}$.   

\begin{verbatim} 
K = ZZ/31991; 
R = K[n,x_1..x_15, y_1..y_15, MonomialOrder => Lex];
h = ideal (x_1-y_1*n^4, x_2-y_2, x_3-y_3*n^4, x_4-y_4, x_5-y_5*n^4, 
n*x_6-y_6,n*x_7-y_7, x_8-y_8, x_9-y_9, n*x_10-y_10, x_11-y_11, x_12-y_12, 
x_13-y_13*n^2, x_14-y_14, x_15-y_15*n^2);
H = gens gb h
Q = entries H
toString Q
\end{verbatim}

\section*{Appendix E}

Computations for the generating set of $\nk$ for $k=1$ from Example \ref{ex3}. The following Macaulay2 computations  produce not only the Hilbert basis $H_\lambda$ of the monoid $\mathcal{M}_\lambda$ but also a generating set of $\mathcal{N}_\lambda^1$. After carefully checking the output, it is easy to see that the following three binomials $x_0x_4x_5-y_0y_4y_5$, $x_0x_2^2x_3^2x_5-y_0y_2^2y_3^2y_5$, 
$x_0x_1-y_0y_1$ are the only ones from the Gr\"obner basis of the specified form in Step 8 of Algorithm \ref{alg3}. In particular, every vector in $\mathcal{N}_\lambda^1$ is of the form $\nu +\mu$ where $\nu \in \{(1,0,0,0,0),(0,0,0,1,1),(0,2,2,0,1) \}$ and
$\mu \in \mathcal{M}_\lambda$.
 Similar computations are used to obtain the generating sets of all other $\nk$. 
\begin{verbatim} 
K = ZZ/31991; 
K[s,t,x_0,x_1..x_5,z_0,z_1..z_5, MonomialOrder => Lex];
h = ideal (s*x_0-z_0, x_1-z_1*s, x_2-z_2*t,s*t*x_3-z_3,s^2*x_4-z_4,x_5-z_5*s^3);
H = gens gb h
Q = entries H
toString Q
\end{verbatim} 

\begin{verbatim}
{{x_4^3*x_5^2-z_4^3*z_5^2, x_2*x_3*z_4^2*z_5-x_4^2*x_5*z_2*z_3, 
x_2*x_3*x_4*x_5-z_2*z_3*z_4*z_5, x_2^2*x_3^2*z_4-x_4*z_2^2*z_3^2, 
x_2^3*x_3^3*x_5-z_2^3*z_3^3*z_5, x_1*z_4*z_5-x_4*x_5*z_1, 
x_1*z_2^2*z_3^2*z_5-x_2^2*x_3^2*x_5*z_1, x_1*x_4*z_2*z_3-x_2*x_3*z_1*z_4, 
x_1*x_4^2*x_5-z_1*z_4^2*z_5, x_1*x_2*x_3-z_1*z_2*z_3, 
x_1^2*z_2*z_3*z_5-x_2*x_3*x_5*z_1^2, x_1^2*x_4-z_1^2*z_4, 
x_1^3*z_5-x_5*z_1^3, x_0*z_4^2*z_5-x_4^2*x_5*z_0, x_0*z_2*z_3-x_2*x_3*z_0, 
x_0*z_1*z_4-x_1*x_4*z_0, x_0*x_5*z_1^2-x_1^2*z_0*z_5, x_0*x_4*x_5-z_0*z_4*z_5, 
x_0*x_2*x_3*z_4-x_4*z_0*z_2*z_3, x_0*x_2*x_3*x_5*z_1-x_1*z_0*z_2*z_3*z_5, 
x_0*x_2^2*x_3^2*x_5-z_0*z_2^2*z_3^2*z_5, x_0*x_1-z_0*z_1, x_0^2*z_4-x_4*z_0^2, 
x_0^2*x_5*z_1-x_1*z_0^2*z_5, x_0^2*x_2*x_3*x_5-z_0^2*z_2*z_3*z_5, 
x_0^3*x_5-z_0^3*z_5, t*z_2-x_2, t*x_3*z_4^2*z_5-x_4^2*x_5*z_3, 
t*x_3*z_1*z_4-x_1*x_4*z_3, t*x_3*z_0-x_0*z_3, t*x_3*x_5*z_1^2-x_1^2*z_3*z_5, 
t*x_3*x_4*x_5-z_3*z_4*z_5, t*x_2*x_3^2*z_4-x_4*z_2*z_3^2, 
t*x_2*x_3^2*x_5*z_1-x_1*z_2*z_3^2*z_5, t*x_2^2*x_3^3*x_5-z_2^2*z_3^3*z_5, 
t*x_1*x_3-z_1*z_3, t*x_0*x_3*z_4-x_4*z_0*z_3, t*x_0*x_3*x_5*z_1-x_1*z_0*z_3*z_5, 
t*x_0*x_2*x_3^2*x_5-z_0*z_2*z_3^2*z_5, t*x_0^2*x_3*x_5-z_0^2*z_3*z_5, 
t^2*x_3^2*z_4-x_4*z_3^2, t^2*x_3^2*x_5*z_1-x_1*z_3^2*z_5, 
t^2*x_2*x_3^3*x_5-z_2*z_3^3*z_5, t^2*x_0*x_3^2*x_5-z_0*z_3^2*z_5, 
t^3*x_3^3*x_5-z_3^3*z_5, s*z_4*z_5-x_4*x_5, s*z_3^2*z_5-t^2*x_3^2*x_5, 
s*z_1-x_1, s*z_0*z_3*z_5-t*x_0*x_3*x_5, s*z_0^2*z_5-x_0^2*x_5, 
s*x_4*z_3-t*x_3*z_4, s*x_4*z_0-x_0*z_4, s*x_4^2*x_5-z_4^2*z_5, 
s*x_2*x_3-z_2*z_3, s*x_1*z_3*z_5-t*x_3*x_5*z_1, s*x_1*z_0*z_5-x_0*x_5*z_1, 
s*x_1*x_4-z_1*z_4, s*x_1^2*z_5-x_5*z_1^2, s*x_0-z_0, s*t*x_3-z_3, 
s^2*z_3*z_5-t*x_3*x_5, s^2*z_0*z_5-x_0*x_5, s^2*x_4-z_4, 
s^2*x_1*z_5-x_5*z_1, s^3*z_5-x_5}} 
\end{verbatim}

\end{document}